\documentclass[]{article}

\usepackage[english]{babel}
\usepackage{amssymb}
\usepackage{amsthm}
\usepackage{amsmath}
\usepackage{mathtools}
\usepackage{tikz-cd}
\usepackage{graphicx}
\usepackage{subcaption}
\usepackage{caption}
\usepackage{pifont}
\usepackage{csquotes}
\usepackage{authblk}
\usepackage[backend=biber,style=numeric,sorting=nyt,maxbibnames=10]{biblatex}
\usepackage[colorlinks, linkcolor=blue]{hyperref}
\theoremstyle{definition}
\newtheorem{Def}{Definition}[section]

\theoremstyle{plain}
\newtheorem{Thm}[Def]{Theorem}

\theoremstyle{definition}

\theoremstyle{plain}
\newtheorem{Lemma}[Def]{Lemma}

\theoremstyle{remark}
\newtheorem{rmk}[Def]{Remark}

\theoremstyle{definition}
\newtheorem{conj}[Def]{Conjecture}

\theoremstyle{plain}
\newtheorem{Corollary}[Def]{Corollary}

\newcommand{\numberset}{\mathbb}

\newcommand{\deb}{\bar{\partial}}
\newcommand{\R}{\numberset{R}}

\newcommand{\C}{\numberset{C}}
\newcommand{\n}{\nabla}

\newcommand{\de}{\partial}

\newcommand{\proj}{\mathbb{P}}

\title{Divisorial properties and special metrics on hypercomplex twistor spaces}
\author{Alberto Pipitone Federico\thanks{Dipartimento di Matematica, Universit\`a di Roma ``Tor Vergata'', Via della Ricerca Scientifica 1, 00133 Roma, Italy \texttt{pipitone@mat.uniroma2.it}}}
\date{\vspace{-3ex}}

\begin{document}
	\maketitle
	\begin{abstract}
		We prove that the general fiber of a compact hypercomplex twistor space with a K\"{a}hler fiber has no divisors nor curves. This is first used to prove that, under the same assumption, the trascendental degree of the field of meromoprhic functions is one. The same result allows to prove that these spaces admit no K\"{a}hler and not even pluriclosed metrics.
	\end{abstract}
	\section*{Introduction}
	The first appearence of twistor spaces dates back to the mathematical physics work by R.  Penrose \cite{penrose1967twistor}, where they are used to transform various equations on a real four dimensional manifold $M$ into holomorphic objects on a complex manifold, the twistor space $Tw(M)$. Since then they have been deeply studied both in physics and geometry. \\
	Among the many generalizations of twistor spaces which have appeared over the years, the most important is arguably in hyperk\"{a}hler geometry. In this context twistor spaces are a powerful tool, since they provide an explicit global deformation $\mathcal{Z} \to \proj^1$. Moreover, if $D$ is the period domain associated with a hyperk\"{a}hler manifold $X$, any two points in $D$ can be connected by a chain of twistor lines, which is an essential argument in the proof of the global Torelli theorem (see \cite{huybrechtsglobal} for details). \\
	The hyperk\"{a}hler twistor construction extends to hypercomplex manifolds, a natural generalization of hyperk\"{a}hler manifolds which first appeared in Obata's work \cite{obata} and has been quite studied since then. Also in this case twistor spaces are strongly related to deformation theory: in \cite{poondeformations}, for example, it is shown that deformations of hypercomplex structures on $X$ correspond to the real deformations (a deformation is real if it preserves the antiholomorphic involution of $Tw(X)$) of the holomorphic twistor map $\pi: Tw(X) \to \proj^1$. \\
	A first interesting and unsolved problem about hypercomplex twistor spaces is to determine if, given a hypercomplex manifold $(X,I,J,K)$ with $(X,I)$ holomorphically symplectic, $(X,I,J,K)$ is itself hyperk\"{a}hler. In \cite{verbitsky2004} it was conjectured that this always holds. \\
	Hypercomplex twistor spaces, though, are also interesting objects on their own, being an example of non projective manifolds with plenty of rational curves (the sections of $\pi: Tw(X) \to \proj^1 $) and
	they have been quite studied from the metric point view. In non-K\"{a}hler geometry, indeed, special metrics have been receiving a lot of attention in the last decades, as some of them have very peculiar properties. Balanced metrics, for example, allows to define the degree of a coherent sheaf as a topological invariant and, as a consequence, to have a well-behaved stability theory.  We refer to \cite{egidi2001special}, \cite{grantcharov} for a survey. For twistor spaces, in particular, in \cite{kaledin1998non} it was proved that the natural metric on a hyperk\"{a}hler twistor space is pluripositive and balanced, whereas in \cite{tomberg2015twistor} balanced metrics over hypercomplex twistor spaces are explicitly constructed. \\
	
	We briefly outline the structure of this work. In the first section we put together Verbitsky's results contained in \cite{verbitsky2004},  \cite{verbitsky2007}, \cite{Soldatenkov_2012} to get Theorem \ref{fiber}. Here, as well as in Corollary \ref{moi}, the assumption of a K\"{a}hler fiber yields a result analogous to the hyperk\"{a}hler case, which might be considered as some evidence for the conjecture cited above. \\
	In relation to Corollary \ref{moi} we point out the recent result by Y. Gorginyan \cite{yulia}, where it is shown that hypercomplex twistor spaces are never Moishezon. In this work we are able to compute exactly the trascendental degree of $Tw(X)$, if one of the fibers is K\"{a}hler.\\
	In the second section we focus on some metric aspects and generalize the results of \cite{verbitsky2014}, proving that a compact hypercomplex twistor spaces admit neither a K\"{a}hler metric, nor a pluriclosed one. 
	In \cite{verbitsky2014}, the non existence of pluriclosed metrics is proved for twistor spaces of four dimensional ASD manifolds. The conclusion there follows from the well known fact that if such a twistor space admits a K\"{a}hler, then it is isomoprhic to $Tw(S^4 )$ or $Tw(\proj^2)$ (see \cite{hitchinkahlerian}). Existence of K\"{a}hler metric for hypercomplex twistor spaces was not known and is the content of Theorem \ref{Kahler}, which again follows from Theorem \ref{fiber}. This allows to prove the non existence of pluriclosed metrics. An alternative proof is also presented, which relies on Y. Gorginyan's result and allows to prove also the non existence of plurinegative metrics.
	\paragraph*{Acknowledgements} I would like to thank Anna Fino and Gueo Grantcharov for having presented me this problem and for their help in the writing process. I also express my gratitude to Antonio Rapagnetta for his support throughout these months.
	\begin{section}{Main definitions and divisorial properties}
		\subsection{Hypercomplex manifolds}
		A hypercomplex manifold is a differentiable manifold $X$ endowed with three complex structures $I,J,K$ satisfying
		\[
		IJ=K \qquad JK=I \qquad KI=J
		\]
		The integrability of $I,J,K$ is equivalent to the existence of a torsion-free connection $\n$, known as Obata connection, satisfying 
		\[
		\n I=\n J =\n K =0
		\]
		A Riemannian metric $g$ is hyperhermitian if it is compatible with $I,J,K$; it is hyperk\"{a}hler if moreover the three fundamental forms $\omega_I, \omega_J, \omega_K$ are closed, or equivalently if the Obata and the Levi-Civita connections coincide. \\
		A holomorphically symplectic manifold is a complex manifold $(X,I)$ admitting a K\"{a}hler metric with an everywhere non-degenerate holomorphic two form $\alpha$. As a consequence of Yau's theorem and Beuaville's decomposition \cite{beauville}, a compact holomorphically symplectic manifold $(X,I)$ together with a K\"{a}hler class $\omega $ gives rise to a hyperk\"{a}hler manifold $(M,g, I,J,K)$, with $\omega=g (-, I-)$. If $X$ is is simply connected, the triple $(I,J,K)$ is uniquely determined.
		\subsection{Twistor spaces and ample rational curves}
		In this paragraph we define the twistor space for a hypercomplex manifold $(X, I,J,K)$. If $(X, I,J,K)$ admits a hyperk\"{a}hler metric $g$, the associated twistor space will be called hyperk\"{a}hler but the metric does not play any role in the construction. The definition in the hypercomplex setting is indeed the most natural one. 
		\begin{Def}
			The twistor space $Tw(X)$ of $(X, I,J,K)$ is the complex manifold $((X \times \proj^1)_{\R}, \mathcal{I})$, where $(-)_{\R}$ denotes the underlying real differentiable manifold and $\mathcal{I}$ is the complex structure\footnote{A proof of the integrability of $\mathcal{I}$ can be found in \cite{kaledin1996integrability}.} given by
			\[
			\mathcal{I}_{(x, \zeta)}(u,v)=(\zeta.u, i v)
			\]
			for any $(x, \zeta)\in X\times  \proj^1 $, $(u,v) \in T_p^{\R} X \times T_{\zeta}^{\R} \proj^1  $.
		\end{Def}
		The first projection $\pi: Tw(X) \to \proj^1$ is a holomorphic map. The second one $p: Tw(X) \to X$ is not, but each fiber of $p$ is a complex submanifold and is the image of a holomorphic section of $\pi$. Each of them has normal bundle $\mathcal{O}(1)^{\oplus \dim(X)} $. We refer to \cite{hitchin1987} for detailed proofs. \\
		We recall that in general a rational curve $C$ in a complex manifold $M$ is said ample if $N_{C,M}$ is ample\footnote{Recall that a vector bundle $E$ is said ample if the tautological line bundle $\mathcal{O}_E(1)$ is ample on $\proj(E)$.} as a vector bundle. Since any vector bundle on $\proj^1$ splits as a direct sum of line bundles, $E  \backsimeq \bigoplus_{i}^{} \mathcal{O}(a_i) $ is ample if and only if $a_i >0$. Equivalently, $C$ is ample if and only if
		\[
		\label{formula1}
		j^* T_{X} \backsimeq  \mathcal{O}(2) \oplus \bigoplus_{i=1}^{ \dim(M)-1} \mathcal{O}(a_i),  \quad a_i >0
		\]
		where $j: C \xhookrightarrow{} X$ is the embedding, thus $j^*T_X$ is ample too. \\
		Ample rational curves enjoy important deformation properties: embedded deformations of $C$ are unobstructed as $H^1(N_{C,X})=0$, thus $\text{Def}(C\subset X)$ is smooth and
		\[
		T_0 \text{Def}(C\subset X) \backsimeq H^0(N_{C,X})
		\]
		The deformations of $C\subset M$, moreover, cover the whole $X$ (see \cite{kollar2013rational} for a proof). \\
		The $a_i$ clearly determine all the deformations property: for twistor spaces $a_i=1$ for all the $i$, thus one can move $C$ keeping at most one point fixed. For this reason $C$ is called a quasi-line, behaving in the same way of a line in $ \proj^n $.	
		\subsection{Divisorial properties of the fibers}
		Let $(X, I,J,K)$ be a compact hypercomplex manifold. The first result of the paper deals with the subvarieties of the fiber of $	\pi: \mathcal{Z}=Tw(X) \to \proj^1$.
		It is a well-known result that if $X$ is hyperk\"{a}hler the general fiber (i.e. over the complement of a countable subset) of $\pi$ contains no curves nor divisors: it was first proved by Fujiki in \cite{fujiki} for the compact case and then by Verbitsky in \cite{verbitsky2004subvarieties} for the general one.
		\begin{rmk}
			\label{hopf}
			The result is not true in general for hypercomplex twistor spaces: let $S$ be the Hopf surface $S = \C^2 \setminus \{ (0,0) \} / (z,w) \backsim (\lambda z, \lambda w)$, with fixed $\lambda \in \R$, $\lambda\neq 0, 1$. Then $S$ has a natural hypercomplex structure, is elliptic and the twistor family is isotrivial (see \cite{pontecorvo1991hermitian} for the details).
		\end{rmk}
		With additional assumption, though, it is possible to recover the same property.
		\begin{Thm}
			\label{fiber}
			Let $\pi: \mathcal{Z} \to \proj^1$ be the twistor projection. If $(X, \lambda)$ admits a K\"{a}hler metric, for some $\lambda\in \proj^1 $, then the general fiber contains no curves nor divisors.
		\end{Thm}
		\begin{proof}
			By \cite{verbitsky2004}, $(X,\lambda)$ is holomorphically symplectic and $(X,I,J,K)$ admits a HKT metric\footnote{A hyperhermitian metric is HKT if $\de_I (\omega_J +i \omega_K)      =0$.}. Under these assumptions, the holonomy of the Obata connection is contained in $\text{SL}(n,\mathbb{H})$, by \cite{verbitsky2007}. Then the claim follows, by \cite{Soldatenkov_2012}, as every analytic subvariety is trianalytic, hence of even codimension.
		\end{proof}
		From this we can describe the divisors on $Tw(X)$, under the assumptions of the existence of a K\"{a}hler fiber. An analogous proof for hyperk\"{a}hler twistor spaces can be found in \cite{verbitsky2012holographyprincipletwistorspaces}.
		\begin{Thm}
			\label{divisor}
			Let $\pi: \mathcal{Z} \to \proj^1$ be the  twistor projection. If one fiber $(X,\lambda)$ admits a K\"{a}hler metric, then any divisor on $\mathcal{Z}$ is a pullback of a divisor on $\proj^1$.
		\end{Thm}
		\begin{proof}
			Let $D$ be an effective divisor. By transversality theorem, the smooth part $ \text{supp}(D)^{sm} \subset \text{supp}(D) $ intersects $\pi^{-1}(\zeta) $ transversally by almost all the $\zeta \in \proj^1 $, in particular the intersection is either empty or has 
			\[
			\text{codim}_{\pi^{-1}(\zeta)}   \big( \pi^{-1}(\zeta) \cap \text{supp}(D)^{sm} \big) \geq 1
			\]
			If the equality holds for a fiber over $\zeta_0$, then $D \cap {\pi^{-1}(\zeta_0)} $ will contain a divisor $D_0$ (it is not excluded that $ \pi^{-1}(\zeta_0)$ contains components of $D$ of higher codimension), since $\text{supp}(D)^{sm}$ is open and dense inside $\text{supp}(D)$). \\
			Thus by Theorem \ref{fiber} we have that for almost all the fibers the above inequality is strict or the intersection is empty. Since $\text{codim}_{\mathcal{Z}} (D)=1 $, though, the only possibility is that $ \pi^{-1}(\zeta) \cap \text{supp}(D)$ is empty for almost all the fibers. \\
			Hence the only remaining possibility is clearly that $\text{supp}(D)$ is the union of some special fibers. This can happen only for finitely many $\zeta$, since by definition $D$ is a divisor and so are all the $\pi^{-1}(\zeta)$. 
		\end{proof}
		The following result is then straightforward.
		\begin{Corollary}\label{moi}
			A compact hypercomplex twistor space $\mathcal{Z}$ admitting one K\"{a}hler fiber $(X, \lambda) $ has $\text{deg}_{\text{tr}}( M_{\mathcal{Z}}  )= 1$.
		\end{Corollary}
		\begin{rmk}
			The twistor space $\mathcal{Z}$ of the Hopf surface $S$ defined in \ref{hopf} is an elliptic fibration over $\proj^1 \times \proj^1$. The isotriviality of the twistor family $q: \mathcal{Z} \to \proj^1$ allows to define a map $\mathcal{Z}\to \proj^1 \times \proj^1 $ by $\zeta \mapsto  (q(\zeta), p(\zeta) )$ where $p:S \to \proj^1 $ is the elliptic fibration of $S$. Then the pullback of $(f,1)$ and $(1,g) $, $f,g \in M_{\proj^1} $ shows that $\text{deg}_{\text{tr}}( M_{\mathcal{Z}}  ) \geq 2$. The equality must hold by Gorginyan's result.
		\end{rmk}
		\subsection{Exotic stuctures}
		It is a well-known fact that small deformations of a holomorphically symplectic manifold $(X,\lambda)$ are holomorphically symplectic and that if $(X,\lambda)$ is simply connected, this holds for all the K\"{a}hler deformations (\cite{beauville}). \\
		In our case it is certainly true that over a small neighborhood of $\lambda\in \proj^1$ the fibers are all holomorphically symplectic, but it is not clear if one can extend this to all the fibers of the family. Also, it is not known if even locally this deformation of $(X,\lambda)$ is isomorphic to a hyperk\"{a}hler twistor deformation. \\
		This is related to the existence of exotic hypercomplex structurues.
		\begin{Def}
			A hypercomplex structure $(X,I,J,K)$ is said exotic if $(X,I)$ is holomorphically symplectic, but for no choice of $\alpha \in \text{Kahl}(X)$ the corresponding hyperk\"{a}hler structure is $(X,I,J,K)$.
		\end{Def} 
		In \cite{verbitsky2004} the following is claimed, and both Theorem \ref{fiber} and Corollary \ref{moi} might be seen as some evidence.
		\begin{conj}
			\label{exoticconj}
			Exotic hypercomplex structures do not exist.
		\end{conj}
	\end{section}
	
	\begin{section}{Metric properties}
		The non k\"{a}hlerianity of hyperk\"{a}hler twistor space was previously known. The two main arguments are based either on the triviality of the canonical of the fibers or on Fujiki's result \cite{fujiki} and neither of these clearly hold in the general hypercomplex setting. We derive the following theorem from Theorem \ref{fiber}.
		\begin{Thm}
			\label{Kahler}
			Let $(X, I,J,K)$ be a compact hypercomplex manifold and $\pi: \mathcal{Z}=Tw(X) \to \proj^1$ the twistor projection. Then $\mathcal{Z}$ admits no K\"{a}hler metrics.
		\end{Thm}
		\begin{proof}
			We have that $H^0(\mathcal{Z}, \Omega^1)=0$: assume that there exists $s\in H^0(\mathcal{Z}, \Omega^1)  $, $s\neq 0$, i.e. a non zero map $T_{  \mathcal{Z}}  \to \mathcal{O}_{\mathcal{Z}} $. Let $j:\proj^1 \to \mathcal{Z}$ be a section of $\pi$, then the restriction $j^*s$ induces a map $ j^* T_{  \mathcal{Z}} \to  j^* \mathcal{O}_{\mathcal{Z}} $ and the first bundle splits as described in \ref{formula1}, thus we have a composition
			\[ 
			\mathcal{O}_{\proj^1} (2) \oplus \mathcal{O}_{\proj^1} (1)^{\oplus \dim(X)} \to j^* \mathcal{O}_{\mathcal{Z}} \xhookrightarrow[]{} \mathcal{O}_{\proj^1}
			\]
			which implies that $j^* s=0$. Since $\mathcal{Z}$ is covered by such sections, $s$ must be zero. \\
			By taking exterior powers it follows that $H^0(\mathcal{Z}, \Omega^i)=0 $, for any $i>1$. In particular $h^{2,0} (\mathcal{Z})=0$, thus if it $\mathcal{Z}$ admits a K\"{a}hler metric,  $\mathcal{Z}$ is projective by Kodaira embedding theorem. This is a contradiction, for Theorem \ref{fiber}.
		\end{proof}
		\begin{rmk}
			The theorem is clearly meaningful only if some fiber of $\pi$ admits a K\"{a}hler metric. In particular if the conjecture \ref{exoticconj} is true, the statement follows.
		\end{rmk}
		Besides K\"{a}hler and balanced metrics, other natural definition of special metrics arise by introducing a notion of positivity on forms.
		\begin{Def}
			A $(k,k)$-form $\alpha$ is stronlgy positive if locally 
			\[
			\alpha = i^k \sum_j \eta_j  \, dz_{j_1} \wedge d\overline{z}_{j_1}  \wedge  \dots \wedge dz_{j_k} \wedge d\overline{z}_{j_k} 
			\]
			with $\eta_j$ positive real numbers; $\alpha$ is said positive if $\alpha \wedge \beta$ is positively oriented, for any strongly positive $(n-k, n-k)$ forms. A form $\alpha$ is (strongly) negative if $-\alpha$ is (strongly) positive. 
		\end{Def}
		\begin{Def}
			A hermitian metric with fundamental form $\omega$ is said 
			\begin{itemize} 
				\item plurinegative (resp. pluripositive) if $\de \deb \omega $ is negative (resp. positive)
				\item pluriclosed if $\de \deb \omega=0$
				\item hermitian symplectic if $\omega$ is the $(1,1)$-part of a closed form
			\end{itemize}
		\end{Def}
		Clearly hermitian symplectic metrics are pluriclosed and pluriclosed metrics are both pluripositive and plurinegative. K\"{a}hler metrics are hermitian symplectic. \\
		We now briefly outline the proof contained in \cite{verbitsky2014}. The same arguments together with Theorem \ref{Kahler} gives the result for hypercomplex twistor spaces. \\
		We recall that for a complex manifold $X$, the Douady space $D_n(X)$ parametrizes pure $n$-dimensional complex subspaces of $X$. Its construction is quite complicated and we refer to \cite{campanapeternell} and \cite{douady} for the details. We only point out that $D_n(X)$ has a complex structure and that there exists a universal family $p:\mathcal{X}\to D_n(X)$, i.e. it is a fine moduli space. \\
		Moreover, a hermitian metric on $X$ induces a natural volume function
		\[
		\text{vol}= p_* \pi^* \omega : D_1(X) \to \R_{\geq 0}
		\]
		where $\pi : \mathcal{X}\to X$ is the natural forgetful map and the pushforward is via integration along fibers. 
		Thus $-\text{vol}$ is plurisubharmonic if $\omega$ is plurinegative\footnote{This holds since (strong) negativity is preserved along pullbacks and for a smooth fiber bundle $E\to B$ with compact fiber one has $ \int_E \alpha \wedge p^* \beta     = \int_B p_* \alpha \wedge \beta  $.}. From this, together with the properness of $\text{vol}$, the following theorem follows.
		\begin{Thm}[Verbitsky, \cite{verbitsky2014}] \label{verb1}
			Let $X$ be a complex manifold admitting a plurinegative metric. Then every connected component of the moduli space of embedded curves is compact.
		\end{Thm}
		The second theorem instead is based on the deformation properties of quasi-lines. All the details can be found in Verbitsky's original paper \cite{verbitsky2014}.
		\begin{Thm}[Verbitsky, \cite{verbitsky2014}] \label{verb2}
			Let $X$ be an $n$ dimensional complex manifold with a smooth rational curve $C \subset X$ with $N_{C,X }\backsimeq \mathcal{O}(1)^{\oplus n-1} $ and let $W$ the connected component of the moduli space of embedded curves containing $C$. If $W$ is compact, then $X$ is Moishezon.
		\end{Thm}
		The result by Y. Gorginyan contained in \cite{yulia} stating that compact hypercomplex twistor spaces are not Moishezon immediately implies the following. 
		\begin{Corollary}
			\label{plurinegative}
			Let $(X, I,J,K)$ be a compact hypercomplex manifold and $\pi: Tw(X) \to \proj^1$ its twistor space. Then $Tw(X)$ admits no plurinegative metrics.
		\end{Corollary}
		The non existence of pluriclosed metrics can be proven, though, also without the above mentioned result, thanks to the first part of this work and to these two additional lemmas.
		\begin{Lemma}[Peternell, \cite{peternellalgebraicity}]\label{peternell}
			If a compact Moishezon manifold $X$ admits a hermitian symplectic metric, then $X$ is K\"{a}hler.
		\end{Lemma}
		\begin{Thm}[Alessandrini, Bassanelli \cite{alessandrini1992positive}]
			\label{A.B.}
			Let $X$ be a compact complex manifold of dimension $n$ for which the $\de \deb$-lemma holds. Then $X$ admits a hermitian symplectic metric if and only if $X$ admits a pluriclosed metric.
		\end{Thm}
		\begin{proof}[Proof of Corollary \ref{plurinegative} for pluriclosed metrics]
			If $Tw(X)$ admits a pluriclosed metric, then $Tw(X)$ is Moishezon by Theorem \ref{verb1} and Theorem \ref{verb2}. Then the $\de \deb$-lemma holds and  $Tw(X)$ admits a  hermitian symplectic metric, by Theorem \ref{A.B.}. The claim follows from Theorem \ref{Kahler} and Lemma \ref{peternell}.
		\end{proof}
		To conclude, Corollary \ref{plurinegative} allows to give an alternative and immediate proof of the result contained in \cite{campana1991twistor} on the twistor spaces of hypercomplex surfaces.
		\begin{Corollary}
			The twistor space of a compact hypercomplex surface is not in Fujiki's class $\mathcal{C}$.
		\end{Corollary}
		\begin{proof}
			In dimension $3$ any manifold in Fujiki's class $\mathcal{C}$ admits a plurinegative metric (see for example \cite{egidi2001special}).
		\end{proof}
	\end{section}
	
	\printbibliography[heading=bibintoc,
	title={Bibliography}
	]

\end{document}